\documentclass[a4paper,11pt,leqno]{smfart}

\usepackage{amssymb,amsmath,enumerate,verbatim,mathrsfs,graphics,graphicx,mathptm,float,mathtools,xcolor,pgf,tikz,epsfig,caption}
\usepackage[T1]{fontenc}
\usepackage[latin1]{inputenc}
\usepackage[english, francais]{babel}
\usepackage[all]{xy}
\usepackage[pdfstartview=FitH,
            colorlinks=true,
            linkcolor=blue,
            urlcolor=blue,
            citecolor=red,
            bookmarks,
            bookmarksopen=true,
            bookmarksnumbered=true]{hyperref}

\usepackage{cleveref}
\theoremstyle{plain}
\newtheorem{theorem}{Th\'eor\`eme}
\newtheorem{theoreme}[theorem]{Th\'eor\`eme}

\theoremstyle{definition}

\theoremstyle{remark}

\theoremstyle{plain}
\newtheorem{thmsec}{Th\'eor\`eme}[section]
\newtheorem{thm}[thmsec]{Th\'eor\`eme}
\newtheorem{pro}[thmsec]{Proposition}
\newtheorem{lem}[thmsec]{Lemme}
\newtheorem{cor}[thmsec]{Corollaire}
\newtheorem{conj}{Conjecture}

\theoremstyle{definition}

\theoremstyle{remark}
\newtheorem{rem}[thmsec]{Remarque}

\def\og{\leavevmode\raise.3ex\hbox{$\scriptscriptstyle\langle\!\langle$~}}
\def\fg{\leavevmode\raise.3ex\hbox{~$\!\scriptscriptstyle\,\rangle\!\rangle$}}

\addtocounter{section}{0}             
\numberwithin{equation}{section}       

\newcommand{\vide}{\emptyset}
\newcommand{\N}{\mathbb{N}}

\newcommand{\Q}{\mathbb{Q}}
\newcommand{\R}{\mathbb{R}}
\newcommand{\C}{\mathbb{C}}
\newcommand{\sph}{\mathbb{P}^{1}_{\mathbb{C}}}
\newcommand{\pp}{\mathbb{P}^{2}_{\mathbb{C}}}
\newcommand{\pd}{\mathbb{\check{P}}^{2}_{\mathbb{C}}}

\newcommand\Sing{\mathrm{Sing}}
\newcommand\Tang{\mathrm{Tang}}
\newcommand\Leg{\mathrm{Leg}}
\newcommand\IF{\mathrm{I}_{\mathcal{F}}}
\newcommand\IinvF{\mathrm{I}_{\mathcal{F}}^{\mathrm{inv}}}
\newcommand\ItrF{\mathrm{I}_{\mathcal{F}}^{\hspace{0.2mm}\mathrm{tr}}}

\newcommand\F{\mathcal{F}}

\newcommand\Ccal{\mathcal{C}}
\newcommand\pref{\mathscr{F}}
\newcommand\preh{\mathscr{H}}
\newcommand\W{\mathcal{W}}

\newcommand\checkC{\hspace{1mm}\check{\hspace{-1mm}\mathcal{C}}}

\newcommand\Omegahesse{\Omega_{\scalebox{0.64}{\ensuremath H}}^{4}}

\newcommand\Omegahilbertcinq{\Omega_{\scalebox{0.64}{\ensuremath H}}^{5}}

\newcommand\Omegahessesept{\Omega_{\scalebox{0.64}{\ensuremath H}}^{7}}

\newcommand\omegaoverline{{\mspace{2mu}\overline{\mspace{-1.4mu}\omega\mspace{-1.4mu}}\mspace{2mu}}}
\newcommand\Omegaoverline{{\mspace{2mu}\overline{\mspace{-1.4mu}\Omega\mspace{-1.4mu}}\mspace{2mu}}}

\newcommand\Hesse{\mathcal{F}_{\hspace{-0.4mm}\raisebox{-0.2mm}{\tiny{$H$}}}^{4}}
\newcommand\Hilbertcinq{\mathcal{F}_{\hspace{-0.4mm}\raisebox{-0.2mm}{\tiny{$H$}}}^{5}}
\newcommand\Hessesept{\mathcal{F}_{\hspace{-0.4mm}\raisebox{-0.2mm}{\tiny{$H$}}}^{7}}

\begin{document}
\title[Le tissu dual d'un pré-feuilletage convexe réduit sur $\mathbb{P}^{2}_{\mathbb{C}}$ est plat]{Le tissu dual d'un pré-feuilletage convexe réduit sur $\mathbb{P}^{2}_{\mathbb{C}}$ est plat}
\date{\today}

\author{Samir \textsc{Bedrouni}}

\address{Facult\'e de Math\'ematiques, USTHB, BP $32$, El-Alia, $16111$ Bab-Ezzouar, Alger, Alg\'erie}
\email{sbedrouni@usthb.dz}


\maketitle{}

\begin{altabstract}\selectlanguage{english}
A holomorphic pre-foliation $\mathscr{F}=\mathcal{C}\boxtimes\mathcal{F}$ on $\mathbb{P}^{2}_{\mathbb{C}}$ is the data of a reduced complex projective curve $\mathcal{C}$ of $\mathbb{P}^{2}_{\mathbb{C}}$ and a holomorphic foliation $\mathcal{F}$ on $\mathbb{P}^{2}_{\mathbb{C}}$. When the foliation $\mathcal{F}$ is convex (resp. reduced convex) and the curve $\mathcal{C}$ is invariant by $\mathcal{F}$, we say that the pre-foliation $\mathscr{F}=\mathcal{C}\boxtimes\mathcal{F}$ is convex (resp. reduced convex). We prove that the dual web of a reduced convex pre-foliation on $\mathbb{P}^{2}_{\mathbb{C}}$ is flat. This generalizes our previous result obtained in the case where the associated curve consists only of invariant lines.

\end{altabstract}

\selectlanguage{french}
\begin{abstract}
Un pré-feuilletage holomorphe $\mathscr{F}=\mathcal{C}\boxtimes\mathcal{F}$ sur $\mathbb{P}^{2}_{\mathbb{C}}$ est la donnée d'une courbe projective complexe réduite $\mathcal{C}$ de $\mathbb{P}^{2}_{\mathbb{C}}$ et d'un feuilletage holomorphe $\mathcal{F}$ sur $\mathbb{P}^{2}_{\mathbb{C}}.$ Lorsque le feuilletage $\mathcal{F}$ est convexe (resp. convexe réduit) et que la courbe $\mathcal{C}$ est invariante par $\mathcal{F}$, on dit que le pré-feuilletage $\mathscr{F}=\mathcal{C}\boxtimes\mathcal{F}$ est convexe (resp. convexe réduit). Nous démontrons que le tissu dual d'un pré-feuilletage convexe réduit sur $\mathbb{P}^{2}_{\mathbb{C}}$ est plat, ce qui généralise un résultat que nous avons obtenu précédemment dans le cas où la courbe associée est composée uniquement de~droites~invariantes.
\noindent{\it Classification math\'ematique par sujets (2010). --- 14C21, 32S65, 53A60.}
\end{abstract}

\section{Introduction}
\bigskip

\noindent Cet article est une continuation de l'étude de la platitude des tissus duaux des pré-feuilletages du plan projectif complexe, initiée dans~\cite{Bed25Dedic} et poursuivie dans \cite{Bed24arxiv}. Nous renvoyons à~\cite{BM18Bull} et \cite{Bed25Dedic} pour les définitions et notations utilisées.

\subsection{Pré-feuilletages convexes sur $\pp$}

Soient $0\leq k\leq d$ des entiers. Un {\sl pré-feuilletage holomorphe $\mathscr{F}$ sur $\pp$ de co-degré $k$ et de degré $d$} est la donnée d'une courbe projective complexe réduite $\mathcal{C}\subset\pp$ de degré~$k$ et d'un feuilletage holomorphe $\F$ sur $\pp$ de degré $d-k.$ On note $\mathscr{F}=\mathcal{C}\boxtimes\F.$ On dit que $\mathcal{F}$ (resp.~$\mathcal{C}$) est le {\sl feuilletage associé} (resp. la {\sl courbe associée}) à $\mathscr{F},$ \emph{voir}~\cite{Bed25Dedic}.

\noindent Les pré-feuilletages de co-degré $0$ et de degré $d$ sont précisément les feuilletages de degré $d$ sur $\pp.$
\smallskip

\noindent Suivant~\cite{MP13} un feuilletage holomorphe sur $\pp$ est dit \textsl{convexe} si ses feuilles, autres que les droites, n'ont pas de points d'inflexion. Notons que, d'après \cite{Per01}, tout feuilletage $\F$ de degré $d\geq1$ sur $\pp$ ne peut avoir plus de $3d$ droites invariantes (distinctes). Lorsque cette borne est atteinte, alors $\F$ est nécessairement convexe; dans ce cas $\F$ est dit \textsl{convexe réduit}.
\smallskip

\noindent Les notions de convexité et de convexité réduite des feuilletages s'étendent de manière naturelle aux pré-feuilletages comme suit: un pré-feuilletage $\pref=\mathcal{C}\boxtimes\F$ sur $\pp$ est dit \textsl{convexe} (resp. \textsl{convexe réduit}) si~le~feuilletage $\F$ est convexe (resp. convexe réduit) et si de plus la courbe $\mathcal{C}$ est invariante par $\F,$ \emph{cf.}~\cite{Bed25Dedic}.

\subsection{Tissus et transformation de \textsc{Legendre}}

\noindent Un $d$-tissu (régulier) $\W$ de $(\mathbb{C}^2,0)$ est la donnée d'une famille $\{\F_1,\F_2,\ldots,\F_d\}$ de feuilletages holomorphes réguliers de $(\mathbb{C}^2,0)$ deux à deux transverses en l'origine. On note $\mathcal{W}=\mathcal{F}_{1}\boxtimes\cdots\boxtimes\mathcal{F}_{d}.$

\noindent Un $d$-tissu (global) sur une surface complexe $M$ est donné dans une carte locale $(x,y)$ par une équation différentielle implicite $F(x,y,y')=0$, où $F(x,y,p)=\sum_{i=0}^{d}a_{i}(x,y)p^{d-i}$ est un polynôme (réduit) en $p$ de degré~$d$, à~coefficients $a_i$ analytiques, avec $a_0$ non identiquement nul. Au voisinage de tout point $z_{0}=(x_{0},y_{0})$ tel que $a_{0}(x_{0},y_{0})\Delta(x_{0},y_{0})\neq 0$, où $\Delta(x,y)$ est le $p$-discriminant de $F$, les courbes intégrales de cette équation définissent un $d$-tissu régulier de $(\mathbb{C}^2,z_{0}).$

\noindent En vertu de~\cite{MP13}, à tout pré-feuilletage $\mathscr{F}=\mathcal{C}\boxtimes\F$ de degré $d\geq1$ et de co-degré $k<d$ sur $\pp$ on~peut associer un $d$-tissu de degré $1$ sur le plan projectif dual $\pd,$ appelé {\sl transformée de \textsc{Legendre}} (ou~tissu dual) de $\pref$ et noté $\Leg\pref$\label{not:Leg-pref}; si $\pref$ est donné dans une carte affine $(x,y)$ de $\pp$ par une $1$-forme $\omega=f(x,y)\left(A(x,y)\mathrm{d}x+B(x,y)\mathrm{d}y\right)$, où $f,A,B\in\mathbb{C}[x,y],$ $\mathrm{pgcd}(A,B)=1,$ alors, dans la carte affine $(p,q)$ de~$\pd$ correspondant à la droite $\{y=px-q\}\subset\pp,$ $\Leg\pref$ est décrit par l'équation différentielle implicite
\[
\check{F}(p,q,x):=f(x,px-q)\left(A(x,px-q)+pB(x,px-q)\right)=0, \qquad \text{avec} \qquad x=\frac{\mathrm{d}q}{\mathrm{d}p}.
\]
Lorsque $k\geq1$, $\Leg\pref$ se décompose en $\Leg\pref=\Leg\mathcal{C}\boxtimes\Leg\F$ où $\Leg\mathcal{C}$ est le $k$-tissu algébrique de $\pd$ défini par l'équation $f(x,px-q)=0$ et $\Leg\F$ est le $(d-k)$-tissu irréductible de degré $1$ de $\pd$ donné par $A(x,px-q)+pB(x,px-q)=0.$

\subsection{Courbure et platitude}

\noindent L'un des premiers résultats en géométrie des tissus, dû à \textsc{Blaschke}-\textsc{Dubourdieu} \cite{BD28}, caractérise l'équivalence locale d'un germe de $3$-tissu régulier $\W=\F_1\boxtimes\F_2\boxtimes\F_3$ de $(\mathbb{C}^2,0)$ avec le $3$-tissu trivial défini par $\mathrm{d}x.\mathrm{d}y.\mathrm{d}(x+y)$ par l'annulation d'une $2$-forme différentielle $K(\mathcal{W})$, appelée courbure de \textsc{Blaschke} de $\mathcal{W}$, et définie comme suit. Pour $i=1,2,3$, soit $\omega_{i}$ une $1$-forme à singularité isolée en $0$ définissant le~feuilletage~$\mathcal{F}_{i}.$ Sans perdre de généralité, on peut supposer que les $1$-formes $\omega_i$ vérifient $\omega_1+\omega_2+\omega_3=0.$ On montre qu'il existe une $1$-forme méromorphe $\eta(\W)$, bien définie à l'addition près d'une $1$-forme fermée logarithmique $\dfrac{\mathrm{d}g}{g}$ avec $g\in\mathcal{O}^*(\mathbb{C}^{2},0)$, telle que $\mathrm{d}\omega_i=\eta(\W)\wedge\omega_i$ pour $i=1,2,3.$ Par définition, la courbure de \textsc{Blaschke} de $\W$ est la $2$-forme $K(\W)=\mathrm{d}\,\eta(\W).$

\noindent Pour un $d$-tissu $\W$ complètement décomposable, $\mathcal{W}=\mathcal{F}_{1}\boxtimes\cdots\boxtimes\mathcal{F}_{d}$ avec $d>3$, on définit la courbure $K(\W)$ de $\W$ comme étant la somme des courbures de \textsc{Blaschke} des sous-$3$-tissus de $\W$.

\noindent On peut vérifier que $K(\mathcal{W})$ est une $2$-forme méromorphe à pôles le long du discriminant $\Delta(\mathcal{W})$ de $\mathcal{W},$ intrinsèquement attachée à $\mathcal{W}.$

\noindent Enfin, si $\mathcal{W}$ est un $d$-tissu sur une surface complexe $M$ (non forcément complètement décomposable), alors on peut le transformer en un $d$-tissu complètement décomposable en prenant son pull-back par un revêtement galoisien ramifié convenable. L'invariance de la courbure de ce nouveau tissu par l'action du groupe de \textsc{Galois} permet de la redescendre en une $2$-forme $K(\W)$ méromorphe globale sur $M,$ à pôles le long du discriminant de $\mathcal{W},$ \emph{voir}~\cite{MP13}.

\noindent Un tissu de courbure nulle est dit plat. Lorsque $M=\pp$ la platitude d'un tissu $\W$ sur $\pp$ se caractérise par l'holomorphie de sa courbure $K(\W)$ le long des points génériques de $\Delta(\W)$.

\subsection{Résultat principal}

Dans~\cite[Théorème~4.2]{MP13} \textsc{Mar\'{\i}n} et \textsc{Pereira} ont montré que le tissu dual d'un feuilletage convexe réduit sur $\pp$ est plat. Dans~\cite[Théorème~E]{Bed25Dedic} nous avons prouvé un résultat analogue pour les pré-feuilletages $\pref=\mathcal{C}\boxtimes\F$ convexes réduits de co-degré $1$ ({\it i.e.} dont la courbe $\mathcal{C}$ est une droite invariante). Dans~\cite[Théorème~1]{Bed24arxiv} nous avons étendu ce résultat au cas où la courbe associée est composée de plusieurs droites invariantes. Dans cet article, nous nous proposons d'établir le théorème suivant, qui traite le cas général d'un pré-feuilletage convexe réduit quelconque.
\begin{theoreme}\label{theoreme:C-invariante-convexe-reduit-plat}
{\sl Soit $\pref=\mathcal{C}\boxtimes\F$ un pré-feuilletage convexe réduit de degré $d\geq3$ sur $\pp.$ Alors le $d$-tissu $\Leg\pref$ est plat.}
\end{theoreme}

\noindent La démonstration de ce théorème fait intervenir plusieurs résultats intermédiaires, chacun ayant son intérêt propre. Nous commençons par démontrer, pour un pré-feuilletage $\pref=\mathcal{C}\boxtimes\F$ sur $\pp$, une formule~(Lemme~\ref{lem:Delta-Leg-Pref}) pour le discriminant $\Delta(\Leg\pref)$ de $\Leg\pref$. Ensuite, nous décrivons localement, près des discriminants $\Delta(\Leg\mathcal{C})$ et $\Delta(\Leg\pref)$ respectivement, le tissu algébrique $\Leg\mathcal{C}$ (Lemmes~\ref{lem:Leg-C-s-check}~et~\ref{lem:Leg-C-C-check}) et le tissu dual $\Leg\pref$ de $\pref$ (Lemmes~\ref{lem:Leg-pref-s-check}~et~\ref{lem:Leg-pref-C-check}). Enfin, lorsque $\pref$ est convexe, nous établissons une caractérisation (Théorème~\ref{thm:Holomorphie-K-Leg-pref-convexe}) de la platitude du tissu $\Leg\pref$, qui, dans le cas particulier où $\pref$ est convexe réduit, se traduit par un résultat (Corollaire~\ref{cor:platitude-Leg-pref-convexe-reduit}) jouant un rôle essentiel dans la preuve.

\section{Rappel sur les singularités et le diviseur d'inflexion d'un feuilletage de $\pp$}\label{sec:singularité-diviseur-inflexion-tissu-dual}
\bigskip

\noindent Un feuilletage holomorphe $\mathcal{F}$ de degré $d$ sur~$\pp$ est défini en coordonnées homogènes $[x:y:z]$ par une $1$-forme du type  $$\omega=a(x,y,z)\mathrm{d}x+b(x,y,z)\mathrm{d}y+c(x,y,z)\mathrm{d}z,$$ o\`{u} $a,$ $b$ et $c$ sont des polynômes homogènes de degré $d+1$ sans facteur commun satisfaisant la condition d'\textsc{Euler} $i_{\mathrm{R}}\omega=0$, où $\mathrm{R}=x\frac{\partial{}}{\partial{x}}+y\frac{\partial{}}{\partial{y}}+z\frac{\partial{}}{\partial{z}}$ désigne le champ radial et $i_{\mathrm{R}}$ le produit intérieur par $\mathrm{R}$. Le {\sl lieu singulier} $\mathrm{Sing}\mathcal{F}$ de $\mathcal{F}$ est le projectivisé du lieu singulier de~$\omega$ $$\mathrm{Sing}\omega=\{(x,y,z)\in\mathbb{C}^3\,\vert \, a(x,y,z)=b(x,y,z)=c(x,y,z)=0\}.$$

\noindent Rappelons quelques notions locales attachées au couple $(\mathcal{F},s)$, où $s\in\Sing\mathcal{F}$. Le germe de $\F$ en $s$ est défini, à multiplication près par une unité de l'anneau local $\mathcal{O}_s$ en $s$, par un champ de vecteurs
\begin{small}
$\mathrm{X}=A(\mathrm{u},\mathrm{v})\frac{\partial{}}{\partial{\mathrm{u}}}+B(\mathrm{u},\mathrm{v})\frac{\partial{}}{\partial{\mathrm{v}}}$.
\end{small}
\noindent La {\sl multiplicité algébrique} $\nu(\mathcal{F},s)$ de $\mathcal{F}$ en $s$ est donnée par $$\nu(\mathcal{F},s)=\min\{\nu(A,s),\nu(B,s)\},$$ où $\nu(g,s)$ désigne la multiplicité algébrique de la fonction $g$ en $s$. L'{\sl ordre de tangence} de $\mathcal{F}$ avec une droite générique passant par $s$ est l'entier $$\tau(\mathcal{F},s)=\min\{k\geq 1\hspace{1mm}\vert\hspace{1mm}\det(J^{k}_{s}\,\mathrm{X},\mathrm{R}_{s})\not\equiv0\}\geq\nu(\mathcal{F},s),$$ où $J^{k}_{s}\,\mathrm{X}$ désigne le $k$-jet de $\mathrm{X}$ en $s$ et $\mathrm{R}_{s}$ le champ radial centré en $s.$
\smallskip

\noindent La singularité $s$ de $\mathcal{F}$ est dite {\sl radiale} si $\nu(\mathcal{F},s)=1$ et si de plus $\tau(\mathcal{F},s)\geq2$. Si tel est le cas, l'entier naturel $\tau(\mathcal{F},s)-1$, compris entre $1$ et $d-1,$ est appelé l'{\sl ordre de radialité} de $s.$

\noindent Dans la suite, nous désignerons par $\Sigma_{\F}^{\mathrm{rad}}$ l'ensemble des singularités radiales de $\F$ et par $\Sigma_{\F}^{\nu\geq2}$ (resp. $\Sigma_{\F}^{\tau\geq2}$) l'ensemble des $s\in\Sing\F$ tels que $\nu(\F,s)\geq2$ (resp. $\tau(\F,s)\geq2$); notons que l'on a $\Sigma_{\F}^{\tau\geq2}=\Sigma_{\F}^{\mathrm{rad}}\cup\Sigma_{\F}^{\nu\geq2}.$
\medskip

\noindent Rappelons la notion du diviseur d'inflexion de $\F$. Soit $\mathrm{Z}=E\frac{\partial}{\partial x}+F\frac{\partial}{\partial y}+G\frac{\partial}{\partial z}$ un champ de vecteurs homogène de degré $d$ sur $\mathbb{C}^3$ non colinéaire au champ radial décrivant $\mathcal{F},$ {\it i.e.} tel que $\omega=i_{\mathrm{R}}i_{\mathrm{Z}}\mathrm{d}x\wedge\mathrm{d}y\wedge\mathrm{d}z.$ Le {\sl diviseur d'inflexion} de $\mathcal{F}$, noté $\IF$, est le diviseur défini par l'équation
\begin{equation*}
\left| \begin{array}{ccc}
x &  E &  \mathrm{Z}(E) \\
y &  F &  \mathrm{Z}(F)  \\
z &  G &  \mathrm{Z}(G)
\end{array} \right|=0.
\end{equation*}

\noindent D'après~\cite{Per01}, $\deg\IF=3d$, et le support de $\IF$ est exactement l'adhérence de l'ensemble des points d'inflexion des feuilles de $\F.$ Plus~précisément, $\IF$ peut se décomposer en $\IF=\IinvF+\ItrF,$ où le support de $\IinvF$ est l'ensemble des droites invariantes par $\mathcal{F}$ et où le support de $\ItrF$ est l'adhérence de l'ensemble des points d'inflexion des feuilles de $\F$ qui ne sont pas des droites.

\noindent Le feuilletage $\mathcal{F}$ est convexe si et seulement si son diviseur d'inflexion $\IF$ est $\F$-invariant, ou, de façon équivalente, si et seulement si $\IF$ est le produit de droites invariantes. Dire que $\F$ est convexe réduit signifie qu'il est convexe à diviseur d'inflexion réduit, autrement dit que
$\IF$ est le produit de $3d$ droites invariantes distinctes.

\section{Discriminant du tissu dual d'un pré-feuilletage sur $\pp$}
\bigskip

\noindent Dans ce paragraphe nous allons établir une formule pour le discriminant du tissu dual d'un pré-feuilletage sur $\pp.$ Pour ce faire, rappelons d'abord que si $\F$ est un feuilletage sur $\pp$, l'application de \textsc{Gauss} est l'application rationnelle $\mathcal{G}_{\F}\hspace{1mm}\colon\pp\dashrightarrow \pd$ définie en tout point régulier $m$~de~$\F$ par $\mathcal{G}_{\F}(m)=\mathrm{T}^{\mathbb{P}}_{m}\F,$ où $\mathrm{T}^{\mathbb{P}}_{m}\F$ désigne la droite tangente à la feuille de $\F$ passant par $m.$ Si $\mathcal{C}\subset\pp$ est une courbe passant par certains points singuliers de~$\F$, on définit $\mathcal{G}_{\mathcal{F}}(\mathcal{C})$ comme étant l'adhérence de $\mathcal{G}_{\F}(\mathcal{C}\setminus\Sing\F).$ Il résulte de \cite[Lemme~2.2]{BFM14} que
\begin{align}\label{equa:Delta-LegF}
\Delta(\Leg\F)
=\mathcal{G}_{\F}(\ItrF)\cup\check{\Sigma}_{\F}^{\tau\geq2}=\mathcal{G}_{\F}(\ItrF)\cup\check{\Sigma}_{\F}^{\mathrm{rad}}\cup\check{\Sigma}_{\F}^{\nu\geq2},
\end{align}
où $\check{\Sigma}_{\F}^{\mathrm{rad}}$, resp. $\check{\Sigma}_{\F}^{\nu\geq2}$, resp. $\check{\Sigma}_{\F}^{\tau\geq2}$, désigne l'ensemble des droites duales des points de $\Sigma_{\F}^{\mathrm{rad}}$, resp. $\Sigma_{\F}^{\nu\geq2}$, resp.~$\Sigma_{\F}^{\tau\geq2}.$

\noindent Soit maintenant $\mathcal{C}$ une courbe projective réduite de $\pp$; notons qu'en vertu de \cite[\S1.4.2]{PP15} nous avons
\begin{align}\label{equa:Delta-LegC}
\Delta(\Leg\mathcal{C})=\checkC\cup\check{\widehat{\Sing\mathcal{C}}},
\end{align}
où $\checkC$ désigne la courbe duale de $\mathcal{C}$ et $\check{\widehat{\Sing\mathcal{C}}}$ l'ensemble des droites duales des points singuliers de $\mathcal{C}.$
\begin{lem}\label{lem:Delta-Leg-Pref}
{\sl
Soit $\mathscr{F}=\mathcal{C}\boxtimes\F$ un pré-feuilletage sur $\pp.$ Notons $\check{\Sigma}_{\F}^{\mathcal{C}}$ l'ensemble des droites duales des points de $\Sigma_{\F}^{\mathcal{C}}:=\Sing\F\cap\mathcal{C}$. Alors
\begin{align}\label{equa:Delta-Leg-Pref}
&\Delta(\Leg\mathscr{F})
=\Delta(\Leg\mathcal{C})\cup\Delta(\Leg\F)\cup\mathcal{G}_{\F}(\mathcal{C})\cup\check{\Sigma}_{\F}^{\mathcal{C}}
=\checkC\cup\mathcal{G}_{\F}(\ItrF\cup\mathcal{C})\cup\check{\widehat{\Sing\mathcal{C}}}\cup\check{\Sigma}_{\F}^{\mathcal{C}}\cup\check{\Sigma}_{\F}^{\tau\geq2}.
\end{align}

\noindent En particulier, si la courbe $\mathcal{C}$ est invariante par $\F$, alors
\begin{align}\label{equa:Delta-Leg-Pref-C-invariante}
&\hspace{-3.84cm}\Delta(\Leg\mathscr{F})
=\Delta(\Leg\F)\cup\checkC\cup\check{\Sigma}_{\F}^{\mathcal{C}}
=\mathcal{G}_{\F}(\ItrF)\cup\checkC\cup\check{\Sigma}_{\F}^{\mathcal{C}}\cup\check{\Sigma}_{\F}^{\tau\geq2}.
\end{align}
}
\end{lem}

\begin{proof}[\sl D\'emonstration]
Nous avons
\begin{align*}
\Delta(\Leg\mathscr{F})=\Delta(\Leg\mathcal{C})\cup\Delta(\Leg\F)\cup\Tang(\Leg\mathcal{C},\Leg\F),
\end{align*}
et un argument de~\cite[page~33]{Bel14} montre que $$\Tang(\Leg\mathcal{C},\Leg\F)=\mathcal{G}_{\F}(\mathcal{C})\cup\check{\Sigma}_{\F}^{\mathcal{C}},$$ d'où la première égalité dans~(\ref{equa:Delta-Leg-Pref}).
La seconde en découle, compte tenu des formules~(\ref{equa:Delta-LegF}) et (\ref{equa:Delta-LegC}).

\noindent Lorsque $\mathcal{C}$ est invariante par $\F$, alors $\mathcal{G}_{\F}(\mathcal{C})=\checkC$ et $\Sing\mathcal{C}\subset\Sigma_{\F}^{\mathcal{C}}$, d'où (\ref{equa:Delta-Leg-Pref-C-invariante}).
\end{proof}

\begin{cor}\label{cor:Delta-Leg-Pref-convexe}
{\sl
Soit $\mathscr{F}=\mathcal{C}\boxtimes\F$ un pré-feuilletage convexe sur $\pp.$ Alors
\begin{align}\label{equa:Delta-Leg-Pref-convexe}
&\Delta(\Leg\mathscr{F})
=\checkC\cup\check{\Sigma}_{\F}^{\mathcal{C}}\cup\check{\Sigma}_{\F}^{\tau\geq2}
=\checkC\cup\check{\Sigma}_{\F}^{\mathcal{C}}\cup\check{\Sigma}_{\F}^{\mathrm{rad}}\cup\check{\Sigma}_{\F}^{\nu\geq2}.
\end{align}

\noindent En particulier, si $\pref$ est convexe réduit, alors
\begin{align}\label{equa:Delta-Leg-Pref-convexe-reduit}
&\hspace{-3.8cm}\Delta(\Leg\mathscr{F})=\checkC\cup\check{\Sigma}_{\F}^{\mathcal{C}}\cup\check{\Sigma}_{\F}^{\mathrm{rad}}.
\end{align}
}
\end{cor}

\begin{proof}[\sl D\'emonstration]
La formule~(\ref{equa:Delta-Leg-Pref-convexe}) résulte de la convexité de $\pref$ et de la formule~(\ref{equa:Delta-Leg-Pref-C-invariante}). Lorsque $\pref$ est convexe réduit, toutes les singularités de $\F$ sont non-dégénérées (\cite[Lemme~6.8]{BM18Bull}) et sont donc de multiplicité algébrique $1$, {\it i.e.} $\Sigma_{\F}^{\nu\geq2}=\vide,$ d'où~(\ref{equa:Delta-Leg-Pref-convexe-reduit}).
\end{proof}

\section{Description des tissus algébriques près du discriminant}
\bigskip

\noindent Dans ce paragraphe, nous décrivons, pour une courbe projective réduite $\mathcal{C}\subset\pp$, le tissu algébrique $\Leg\mathcal{C}$ près d'une composante irréductible de son discriminant $\Delta(\Leg\mathcal{C})$. Le lemme suivant traite le cas de la droite~$\check{s}$ duale d'un point $s\in\mathcal{C}$ ($\check{s}\subset\Delta(\Leg\mathcal{C})\Longleftrightarrow s\in\Sing\mathcal{C}$).

\begin{lem}\label{lem:Leg-C-s-check}
{\sl
Soient $\mathcal{C}\subset\pp$ une courbe projective réduite de degré $k$ et $s$ un point de $\mathcal{C}$ de multiplicité algébrique $n.$ Alors, au voisinage d'un point générique de la droite $\check{s}$ duale de $s$, le $k$-tissu algébrique $\Leg\mathcal{C}$ peut se décomposer en
\begin{equation}\label{equa:Leg-C-s-check}
\Leg\mathcal{C}=\W_n\boxtimes\W_{k-n},
\end{equation}
où $\W_n$ est un $n$-tissu admettant $\check{s}$ comme courbe totalement invariante et $\W_{k-n}$ est un $(k-n)$-tissu régulier et transverse à $\check{s}.$}
\end{lem}

\begin{proof}[\sl D\'emonstration]
Fixons une carte affine $(x,y)$ de $\pp$ telle que $s=(0,0)$ et soit $P(x,y)=0$ une équation affine de $\mathcal{C}.$ Dans la carte affine $(p,q)$ de $\pd$ associée à la droite $\{y=px-q\}\subset\pp,$ on a $\check{s}=\{q=0\}$ et le $k$-tissu $\Leg\mathcal{C}$ est décrit par l'équation différentielle $F(p,q,x):=P(x,px-q)=0$ avec $x=\frac{\mathrm{d}q}{\mathrm{d}p}.$ L'égalité $\nu(\mathcal{C},s)=n$ entraîne que $F\in\C[p,q,x]$ est de la forme
\begin{align*}
F(p,q,x)=a_{0}(q)q^n+a_{1}(p,q)q^{n-1}x+\cdots+a_{n-1}(p,q)q\hspace{0.2mm}x^{n-1}+a_{n}(p,q)x^{n}+\cdots+a_{k}(p,q)x^{k},
\end{align*}
avec $a_{n}(p,0)\not\equiv0.$ Par application du théorème de préparation de \textsc{Weierstrass} on peut écrire
\begin{align*}
F(p,q,x)=U(p,q,x)\left(x^{n}-q\left(\bar{a}_{n-1}(p,q)x^{n-1}+\cdots+\bar{a}_{0}(p,q)\right)\right),
\end{align*}
avec $U(p,0,0)\not\equiv0.$ Ainsi, au voisinage d'un point générique de $\check{s}=\{q=0\},$ $\Leg\mathcal{C}=\W_n\boxtimes\W_{k-n},$ où $\W_{k-n}$ est un $(k-n)$-tissu transverse à $\check{s}$ et $\check{s}$ est totalement invariante par $\W_{n}.$ Les feuilles de $\W_{k-n}$ étant des morceaux de droites, la transversalité de $\W_{k-n}$ à $\check{s}$ implique sa régularité près de $\check{s}.$
\end{proof}

\begin{rem}\label{rem:point-ordinaire}
Lorsque $s$ est un point multiple ordinaire de $\mathcal{C}$ de multiplicité $n$, {\it i.e.} si $\mathcal{C}$ admet $n$ tangentes différentes par le point $s$, alors le $n$-tissu $\W_n$ est complètement décomposable et sa courbure est holomorphe le long de $\check{s}.$

\noindent En effet, si $P(x,y)=\prod_{i=1}^{n}(\alpha_i\,x+\beta_i\,y)+\cdots$, avec $[\alpha_i:\beta_i]\neq[\alpha_j:\beta_j]$ si $i\neq j$, alors, au voisinage d'un point générique de $\check{s}$, le tissu $\W_n$ se décompose en $\W_n=\boxtimes_{i=1}^{n}\F_{i},$ où~$\F_{i}:\mathrm{d}q-q\left(\frac{\beta_i}{\alpha_i+\beta_i\,p}+\cdots\right)\mathrm{d}p=0.$ Le~Lemme~2.1~de~\cite{Bed24arxiv} assure que
\begin{align*}
K(\Leg\mathcal{C})=K(\W_n)-(n-2)\sum_{i=1}^{n}K(\F_i\boxtimes\W_{k-n})+\sum_{1\leq i<j\leq n}K(\F_i\boxtimes\F_j\boxtimes\W_{k-n})+\binom{n-1}{2}K(\W_{k-n}).
\end{align*}
Comme $\W_{k-n}$ (resp. $\F_i\boxtimes\W_{k-n}$) est régulier près de $\check{s}$, $K(\W_{k-n})$  (resp. $K(\F_i\boxtimes\W_{k-n})$) est holomorphe sur~$\check{s}.$ La courbure de $\F_i\boxtimes\F_j\boxtimes\W_{k-n}$ est holomorphe le long de $\check{s}$ par application de \cite[Théorème~1]{MP13}. Puisque $K(\Leg\mathcal{C})\equiv0$, il en résulte que $K(\W_n)$ est holomorphe sur $\check{s}.$
\end{rem}

\noindent Le lemme suivant donne une description du tissu $\Leg\mathcal{C}$ près de la courbe $\checkC_{0}$ duale d'une composante irréductible $\mathcal{C}_{0}$ de $\mathcal{C}$ qui n'est pas une droite.

\begin{lem}\label{lem:Leg-C-C-check}
{\sl
Soit $\mathcal{C}\subset\pp$ une courbe projective réduite de degré $k\geq2.$ On suppose que $\mathcal{C}$ n'est pas une union de droites et soit $\mathcal{C}_{0}$ une composante irréductible de $\mathcal{C}$ de degré supérieur ou égal à $2.$ Alors, la courbe duale $\checkC_{0}$ est invariante par $\Leg\mathcal{C}.$ De plus, au voisinage d'un point générique de $\checkC_{0}$, le $k$-tissu $\Leg\mathcal{C}$ peut se décomposer en
\begin{equation}\label{equa:Leg-C-C-check}
\Leg\mathcal{C}=\W_2\boxtimes\W_{k-2},
\end{equation}
où $\W_2$ est un $2$-tissu ayant $\checkC_{0}$ comme courbe totalement invariante et dont la multiplicité du discriminant $\Delta(\W_{2})$ le long de $\checkC_{0}$ est minimale, égale à $1,$ et $\W_{k-2}$ est un $(k-2)$-tissu régulier et transverse~à~$\checkC_{0}.$
}
\end{lem}

\begin{proof}[\sl D\'emonstration]
Le sous-tissu $\Leg\mathcal{C}_{0}$ étant formé des droites tangentes à $\checkC_{0}$, nous avons $\mathrm{T}_{m}\checkC_{0}\subset\mathrm{T}_{m}\Leg\mathcal{C}_{0}$ pour tout point régulier $m$ de $\checkC_{0},$ ce qui montre que $\checkC_{0}$ est invariante par $\Leg\mathcal{C}_{0}$ et donc par $\Leg\mathcal{C}.$

\noindent Notons $\W(m)$ le germe du $k$-tissu $\Leg\mathcal{C}$ en un point générique $m\in\checkC_{0}.$ D'après \cite[\S1.4.2]{PP15}, nous avons $\mathrm{mult}(\Delta(\W(m)),\checkC_{0})=1$; si $k=2$ alors $\checkC_{0}=\checkC$ est totalement invariante par $\W(m)$ et $\W(m)=\W_2$ est irréductible par application de~\cite[Lemme~2.5]{MP13}.

\noindent Supposons $k\geq3$; alors $\W(m)$ n'est pas irréductible, car sinon $\checkC_{0}$ serait totalement invariante par $\W(m),$ de sorte que (\emph{cf.}~\cite[Lemme~2.5]{MP13}) $1=\mathrm{mult}(\Delta(\W(m)),\checkC_{0})\geq k-1>1$, ce qui est impossible. Nous pouvons donc écrire $\W(m)=\W_n\boxtimes\W_{k-n},$ où $\W_n$ est un $n$-tissu ayant $\checkC_{0}$ comme courbe totalement invariante, avec $n\geq2,$ et $\W_{k-n}$ est un $(k-n)$-tissu régulier et transverse à $\checkC_{0}.$ Alors, toujours en vertu~de~\cite[Lemme~2.5]{MP13}, nous avons $1=\mathrm{mult}(\Delta(\W_n),\checkC_{0})\geq n-1,$ d'où $n=2$; ainsi $\W(m)=\W_2\boxtimes\W_{k-2}$.
\end{proof}

\section{Description du tissu dual d'un pré-feuilletage convexe sur $\pp$ près de son discriminant}
\bigskip

\noindent Nous nous intéressons ici à la description du tissu dual $\Leg\pref$ d'un pré-feuilletage $\pref=\F\boxtimes\mathcal{C}$ sur $\pp$ près de certaines composantes~irréductibles~du~discriminant $\Delta(\Leg\pref)$, en nous focalisant sur le cas où $\pref$~est~convexe. Nous commençons par traiter le cas de la droite~$\check{s}$ duale d'une singularité $s$ de $\F$ sur $\mathcal{C}$ de multiplicité algébrique~$1.$

\begin{lem}\label{lem:Leg-pref-s-check}
{\sl
Soit $\pref=\mathcal{C}\boxtimes\F$ un pré-feuilletage de degré $d$ sur $\pp.$ Supposons que $\F$ ait une singularité~$s$ sur $\mathcal{C}$ de multiplicité algébrique $1$ ({\it i.e.} $\nu(\F,s)=1$). Désignons par $n$ la multiplicité algébrique de $\mathcal{C}$~en~$s$ et~par~$\tau$ l'ordre de tangence de $\F$ avec une droite générique passant par $s$ ({\it i.e.} $\tau:=\tau(\F,s)$). Alors, au voisinage d'un point générique $m$ de la droite $\check{s}$ duale de $s$, le~$d$-tissu $\Leg\pref$ peut se décomposer en
\begin{equation}\label{equa:Leg-pref-s-check-n-geq-0}
\Leg\pref=\W_n\boxtimes\W_{\tau}\boxtimes\W_{d-n-\tau},
\end{equation}
où $\W_n$ est un $n$-tissu admettant $\check{s}$ comme courbe totalement invariante, $\W_{\tau}$ est un $\tau$-tissu laissant $\check{s}$ totalement invariante et tel que $\mathrm{mult}(\Delta(\W_{\tau}),\check{s})=\tau-1,$ et $\W_{d-n-\tau}$ est un $(d-n-\tau)$-tissu transverse à $\check{s}.$ De plus si $\pref$ est convexe alors $\W_{d-n-\tau}$ est~régulier au voisinage de $m.$
}
\end{lem}

\begin{proof}[\sl D\'emonstration]
Posons $k:=\deg\mathcal{C}.$ D'après~\cite[Proposition~3.3]{MP13}, au voisinage de $m$, nous pouvons décomposer le tissu $\Leg\F$ sous la forme $\Leg\F=\W_{\tau}\boxtimes\W_{d-k-\tau},$ où $\W_{\tau}$ est un $\tau$-tissu ayant $\check{s}$ comme courbe totalement invariante, avec $\mathrm{mult}(\Delta(\W_{\tau}),\check{s})=\tau-1,$ et $\W_{d-k-\tau}$ est un $(d-k-\tau)$-tissu transverse~à~$\check{s}.$ Par~ailleurs,~au~voisinage~de~$m$, nous pouvons écrire $\Leg\mathcal{C}=\W_n\boxtimes\W_{k-n},$ où les tissus $\W_n$ et $\W_{k-n}$ sont comme dans le Lemme~\ref{lem:Leg-C-s-check}. Alors $\W_{d-n-\tau}:=\W_{k-n}\boxtimes\W_{d-k-\tau}$ est transverse à $\check{s}$ et, au voisinage de $m$, $\Leg\pref$ admet la décomposition~(\ref{equa:Leg-pref-s-check-n-geq-0}).
\vspace{1mm}

\noindent Supposons que $\pref$ soit convexe. Alors, la convexité de $\F$ implique, par un argument de la démonstration de~\cite[Théorème~4.2]{MP13}, que le tissu $\W_{d-k-\tau}$ est régulier au voisinage de $m.$ Comme les feuilles de $\W_{k-n}$ sont des morceaux de droites tangentes à $\Leg\F$ le long de $\checkC$, l'égalité $\deg(\Leg\F)=1$ entraîne que $\Tang(\W_{k-n},\W_{d-k-\tau})=\emptyset.$ Il en résulte que $\W_{d-n-\tau}$ est régulier au voisinage de $m.$
\end{proof}

\begin{rem}\label{rem:Leg-pref-s-check-n=0}
Lorsque $s\not\in\mathcal{C}$, {\it i.e.} si $n=0$, alors $\Leg\mathcal{C}=\W_{k}$ est régulier et transverse à $\check{s}$ et nous avons
\begin{equation}\label{equa:Leg-pref-s-check-n=0}
\Leg\pref=\W_{\tau}\boxtimes\W_{d-\tau},
\end{equation}
où $\W_{d-\tau}:=\W_{k}\boxtimes\W_{d-k-\tau}$ est transverse à $\check{s}.$
\end{rem}

\noindent Maintenant, nous allons démontrer le lemme suivant, qui décrit le tissu $\Leg\pref=\Leg\mathcal{C}\boxtimes\Leg\F$ près de la courbe $\checkC_{0}$ duale d'une composante irréductible $\mathcal{C}_{0}$ de $\mathcal{C}$ qui n'est pas une droite.

\begin{lem}\label{lem:Leg-pref-C-check}
{\sl
Soit $\pref=\mathcal{C}\boxtimes\F$ un pré-feuilletage convexe de degré $d\geq3$ sur $\pp$, dont la courbe associée $\mathcal{C}$ n'est pas une union de droites. Soit $\mathcal{C}_{0}$ une composante irréductible de $\mathcal{C}$ de degré supérieur ou égal à $2.$ Alors, au voisinage d'un point générique de $\checkC_{0}$, le $d$-tissu $\Leg\pref$ peut se décomposer en
\begin{equation}\label{equa:Leg-pref-C-check}
\Leg\pref=\F_{0}\boxtimes\W_{2}\boxtimes\W_{d-3},
\end{equation}
où $\W_{d-3}$ est un $(d-3)$-tissu régulier et transverse à $\checkC_{0},$ $\W_2$ est un $2$-tissu ayant $\checkC_{0}$ comme courbe totalement invariante et tel que $\mathrm{mult}(\Delta(\W_2),\checkC_{0})=1$, et $\F_{0}$ est un feuilletage laissant $\checkC_{0}$ invariante.
}
\end{lem}

\begin{proof}[\sl D\'emonstration]
Comme $\mathcal{C}$ est $\F$-invariante par hypothèse, chaque composante irréductible de $\checkC$, et en particulier $\checkC_{0}$, est $\Leg\F$-invariante, \emph{voir}~\cite[page~175]{MP13}. La convexité de $\F$ et la formule~(\ref{equa:Delta-LegF}) entraînent que $\Delta(\Leg\F)=\check{\Sigma}_{\F}^{\tau\geq2}$, de sorte que $\checkC_{0}\not\subset\Delta(\Leg\F).$ Par suite, au voisinage d'un point générique $m$~de~$\checkC_{0}$, on~peut~écrire~$\Leg\F=\F_0\boxtimes\W_{d-k-1}$, où $\F_0$ est un feuilletage laissant $\checkC_{0}$ invariante et $\W_{d-k-1}$ est un $(d-k-1)$-tissu régulier et transverse à $\checkC_{0},$ avec $k:=\deg\mathcal{C}.$ D'autre part, le $k$-tissu $\Leg\mathcal{C}$ se décompose au voisinage~de~$m$ sous la forme
$\Leg\mathcal{C}=\W_2\boxtimes\W_{k-2},$ où les tissus $\W_2$ et $\W_{k-2}$ sont comme~dans~le~Lemme~\ref{lem:Leg-C-C-check}. Alors~$\Leg\pref=\F_{0}\boxtimes\W_{2}\boxtimes\W_{d-3}$, où $\W_{d-3}:=\W_{k-2}\boxtimes\W_{d-k-1}$ est transverse à $\checkC_{0}.$ Les feuilles de $\W_{k-2}$ étant des morceaux de droites tangentes à $\Leg\F$ le long de $\checkC,$ l'égalité $\deg(\Leg\F)=1$ implique que $\Tang(\W_{k-2},\W_{d-k-1})=\emptyset$ et donc que $\W_{d-3}$ est régulier au voisinage de $m.$ Le lemme est ainsi démontré.

\end{proof}

\section{Caractérisation de la platitude du tissu dual d'un pré-feuilletage convexe sur $\pp$}
\bigskip

\noindent Dans cette section, nous démontrons un théorème qui caractérise la platitude du tissu dual d'un pré-feuilletage convexe $\pref$ sur $\pp$. Nous le particularisons ensuite au cas où $\pref$ est convexe réduit.

\begin{thm}\label{thm:Holomorphie-K-Leg-pref-convexe}
{\sl
Soit $\pref=\mathcal{C}\boxtimes\F$ un pré-feuilletage convexe de degré $d\geq3$ sur $\pp.$ Posons $\Xi_{\F}^{\mathcal{C}}:=\Sing\mathcal{C}\setminus\Sigma_{\F}^{\mathrm{rad}}$ et notons $\check{\Xi}_{\F}^{\mathcal{C}}$ l'ensemble des droites duales des points de $\Xi_{\F}^{\mathcal{C}}.$ Alors la courbure de $\Leg\pref$ est holomorphe sur $\pd\setminus\big(\check{\Xi}_{\F}^{\mathcal{C}}\cup\check{\Sigma}_{\F}^{\nu\geq2}\big).$ En particulier, le $d$-tissu $\Leg\pref$ est plat si et seulement si $K(\Leg\pref)$ est holomorphe~le~long~de~$\check{\Xi}_{\F}^{\mathcal{C}}\cup\check{\Sigma}_{\F}^{\nu\geq2}.$
}
\end{thm}

\noindent La démonstration de ce théorème passe par deux autres énoncés.

\begin{pro}\label{pro:holomorphie-courbure-W-n-W-tau-W-d-tau-n}
{\sl Soit $\W_{\tau}$ un germe de $\tau$-tissu de $(\mathbb{C}^{2},0),\,\tau\geq2.$ Supposons que $\Delta(\W_{\tau})$ possède une composante irréductible $C$ totalement invariante par $\W_{\tau}$ et de multiplicité minimale $\tau-1.$ Soit $\W_{n}=\F_1\boxtimes\cdots\boxtimes\F_n$~un~germe de $n$-tissu de $(\mathbb{C}^{2},0)$ complètement décomposable laissant $C$ totalement invariante. Soit~$\W_{d-n-\tau}$~un~germe~de~$(d-n-\tau)$-tissu régulier de $(\mathbb{C}^{2},0)$ transverse à $C.$ Alors la courbure du $d$-tissu $\W=\W_{n}\boxtimes\W_{\tau}\boxtimes\W_{d-n-\tau}$ est holomorphe le long de $C$ si et seulement si la courbure du $n$-tissu $\W_{n}$ est holomorphe le long~de~$C.$
}
\end{pro}

\noindent Cette proposition généralise~\cite[Proposition~2.6]{MP13} et \cite[Proposition~3.9~et~Remarque~3.10]{Bed25Dedic}.

\begin{proof}[\sl D\'emonstration]
En posant $\W_{d-n}:=\W_{\tau}\boxtimes\W_{d-n-\tau}$, \cite[Lemme~2.1]{Bed24arxiv} entraîne que
\begin{align*}
K(\W)=K(\W_n)-(n-2)\sum_{i=1}^{n}K(\F_i\boxtimes\W_{d-n})+\sum_{1\leq i<j\leq n}K(\F_i\boxtimes\F_j\boxtimes\W_{d-n})+\binom{n-1}{2}K(\W_{d-n}).
\end{align*}
Or, $K(\W_{d-n})$, resp. $K(\F_i\boxtimes\W_{d-n})$, resp. $K(\F_i\boxtimes\F_j\boxtimes\W_{d-n})$, est holomorphe le long de $C$, en vertu de \cite[Proposition~2.6]{MP13}, resp.~\cite[Proposition~3.9]{Bed25Dedic}, resp.~\cite[Remarque~3.10]{Bed25Dedic}. D'où l'équivalence entre l'holomorphie sur $C$ de $K(\W)$ et celle de $K(\W_n).$
\end{proof}

\begin{lem}\label{lem:radiale-ordinaire}
{\sl Soit $\F$ un feuilletage sur $\pp$ ayant une courbe algébrique invariante réduite $\mathcal{C}.$ Supposons que $\F$ possède une singularité radiale $s$ sur $\mathcal{C}.$ Alors $s$ est un point multiple ordinaire de $\mathcal{C}.$}
\end{lem}

\begin{proof}[\sl D\'emonstration]
Choisissons des coordonnées affines $(x,y)$ telles que $s=(0,0)$ et soit $P(x,y)=0$ une équation affine de $\mathcal{C}.$ Soit $P=\prod_{i=1}^{r}f_{i}^{\,k_i}$ une décomposition de $P\in\C[x,y]$ en facteurs irréductibles dans $\C[[x,y]]$; comme $P$ est par hypothèse réduit, nous avons $k_i=1$ pour tout $i\in\{1,\ldots,r\},$ \emph{cf.}~\cite[Chapitre~6]{Chen78}. Par ailleurs, le fait que~$s\in\Sigma_{\F}^{\mathrm{rad}}$ entraîne, d'après~\cite[Théorème~8]{Sei68}, que $f_i(x,y)=a_i\,x+b_i\,y+\cdots$, avec $[a_i:b_i]\in\sph$~et~$[a_i:b_i]\neq[a_j:b_j]$ si $i\neq j.$ Il en résulte que la partie homogène de plus bas degré de $P$ est~égale~à~$\prod_{i=1}^{r}(a_i\,x+b_i\,y),$ ce qui montre que $s$ est un point multiple ordinaire de $\mathcal{C}$ de multiplicité $r.$
\end{proof}

\begin{proof}[\sl D\'emonstration du Théorème~\ref{thm:Holomorphie-K-Leg-pref-convexe}]
D'abord, soit $\mathcal{C}_0$ une composante irréductible de $\Ccal$ qui n'est pas une droite; le $d$-tissu $\Leg\pref$ se décompose près de $\checkC_{0}$ sous la forme (\ref{equa:Leg-pref-C-check}). Ainsi $K(\Leg\pref)$ est holomorphe le~long~de~$\checkC_{0}$ par application~de~\cite[Proposition~3.9]{Bed25Dedic}.

\noindent Ensuite, soit $s\in\Sigma_{\F}^{\mathrm{rad}}.$ Si $s\not\in\mathcal{C}$, alors, pour montrer que $K(\Leg\pref)$ est holomorphe sur $\check{s}$, il suffit de décomposer $\Leg\pref$ près de $\check{s}$ sous la forme~(\ref{equa:Leg-pref-s-check-n=0}) et d'appliquer~\cite[Proposition~2.6]{MP13}. Supposons donc $s\in\mathcal{C}.$ D'après le Lemme~\ref{lem:radiale-ordinaire}, le point $s$ est un point multiple ordinaire de $\mathcal{C}.$ De plus, le tissu $\Leg\pref$ se décompose près de $\check{s}$ sous la forme (\ref{equa:Leg-pref-s-check-n-geq-0}). Il résulte alors de la Proposition~\ref{pro:holomorphie-courbure-W-n-W-tau-W-d-tau-n} et de la Remarque~\ref{rem:point-ordinaire} que $K(\Leg\pref)$ est holomorphe~sur~$\check{s}.$

\noindent Maintenant, pour tout $s$ de $\Sigma_{\F}^{\mathcal{C}}\setminus\big(\Sing\Ccal\cup\Sigma_{\F}^{\mathrm{rad}}\cup\Sigma_{\F}^{\nu\geq2}\big)
=\Sigma_{\F}^{\mathcal{C}}\setminus\big(\Sing\Ccal\cup\Sigma_{\F}^{\tau\geq2}\big)$, on peut décomposer $\Leg\pref$ près de $\check{s}$ sous la forme (\ref{equa:Leg-pref-s-check-n-geq-0}) avec $n=\tau=1$. Par~suite,~$K(\Leg\pref)$ est holomorphe le long de $\check{s}$ par application de~\cite[Théorème~1]{MP13}.

\noindent Enfin, d'après la formule~(\ref{equa:Delta-Leg-Pref-convexe}), il en résulte que $K(\Leg\pref)$ est holomorphe sur $\Delta(\Leg\mathscr{F})\setminus\big(\check{\Xi}_{\F}^{\mathcal{C}}\cup\check{\Sigma}_{\F}^{\nu\geq2}\big)$ et~donc~sur~$\pd\setminus\big(\check{\Xi}_{\F}^{\mathcal{C}}\cup\check{\Sigma}_{\F}^{\nu\geq2}\big),$ d'où le théorème.
\end{proof}

\begin{cor}\label{cor:platitude-Leg-pref-convexe-reduit}
{\sl
Soit $\pref=\mathcal{C}\boxtimes\F$ un pré-feuilletage convexe réduit de degré $d\geq3$ sur $\pp.$ Alors la courbure de~$\Leg\pref$~est~holomorphe~sur~$\pd\setminus\check{\Xi}_{\F}^{\mathcal{C}}.$ En particulier, le $d$-tissu $\Leg\pref$ est plat si~et~seulement~si $K(\Leg\pref)$ est holomorphe~le~long~de~$\check{\Xi}_{\F}^{\mathcal{C}}.$
}
\end{cor}

\begin{proof}[\sl D\'emonstration]
\noindent La convexité réduite de $\F$ assure que $\Sigma_{\F}^{\nu\geq2}=\vide$ (\emph{cf.} preuve du Corollaire~\ref{cor:Delta-Leg-Pref-convexe}); il suffit alors d'appliquer le Théorème~\ref{thm:Holomorphie-K-Leg-pref-convexe}.
\end{proof}

\bigskip

\section{Preuve du Théorème~\ref{theoreme:C-invariante-convexe-reduit-plat}}\label{sec:preuve-theoreme-1}
\bigskip

\noindent Avant de démontrer le Théorème~\ref{theoreme:C-invariante-convexe-reduit-plat}, nous avons besoin de la proposition suivante.

\begin{pro}\label{pro:F-convexe-reduit-singularite-simple}
{\sl Soit $\F$ un feuilletage convexe réduit sur $\pp$.

\noindent\textbf{\textit{1.}} Par toute singularité non radiale $m$ de $\F$ passent exactement deux droites $\ell_{m}^{(1)}$ et $\ell_{m}^{(2)}$ invariantes par $\F$.

\noindent\textbf{\textit{2.}} Pour tout $m\in\Sing\F\setminus\Sigma_{\F}^{\mathrm{rad}}$ et pour tout $i\in\{1,2\}$, nous avons $\mathrm{CS}(\F,\ell_{m}^{(i)},m)\not\in\R^{+}$. En particulier, toute singularité non radiale de $\F$ est simple.

\noindent\textbf{\textit{3.}} Soit $\mathcal{C}\subset\pp$ une courbe irréductible invariante par $\F$ qui n'est pas une droite. Alors toutes les singularités de $\F$ sur $\mathcal{C}$ sont radiales, {\it i.e.} $\Sigma_{\F}^{\mathcal{C}}\subset\Sigma_{\F}^{\mathrm{rad}}.$
}
\end{pro}

\begin{proof}[\sl D\'emonstration]
Le premier point est tiré de \cite[Lemme~3.1]{BM20Z}. Montrons le deuxième point. Soit~$m\in\Sing\F\setminus\Sigma_{\F}^{\mathrm{rad}}$. D'après \cite[Proposition~3.2]{BM20Z}, pour $i=1,2$, il existe un feuilletage homogène convexe $\mathcal{H}_{m}^{(i)}$ sur $\pp$ appartenant à l'adhérence de \textsc{Zariski} de la $\mathrm{Aut}(\pp)$-orbite de $\F$ et vérifiant les propriétés suivantes
\begin{itemize}
\item [$\bullet$] la droite $\ell_{m}^{(i)}$ est invariante par $\mathcal{H}_{m}^{(i)}$;
\vspace{0.5mm}

\item [$\bullet$] le point $m$ est singulier non-dégénéré et non-radial pour $\mathcal{H}^{(i)}_{m}$;
\vspace{0.5mm}

\item [$\bullet$] $\mathrm{CS}(\mathcal{H}^{(i)}_{m},\ell_{m}^{(i)},m)=\mathrm{CS}(\F,\ell_{m}^{(i)},m)\neq0$.
\end{itemize}
\vspace{1mm}

\noindent Par ailleurs, pour $i=1,2$, \cite[Corollaire~3.3]{BM21Publ} assure que $\mathrm{Re}\big(\mathrm{CS}(\mathcal{H}^{(i)}_{m},\ell_{m}^{(i)},m)\big)<\frac{1}{2}.$ Maintenant, si l'on avait~$\mathrm{CS}(\F,\ell_{m}^{(1)},m)\in\R^{+}$, on aurait aussi $\mathrm{CS}(\F,\ell_{m}^{(2)},m)\in\R^{+}$, car $\mathrm{CS}(\F,\ell_{m}^{(1)},m)\mathrm{CS}(\F,\ell_{m}^{(2)},m)=1.$ On~en~déduirait~que~$1=\mathrm{CS}(\F,\ell_{m}^{(1)},m)\mathrm{CS}(\F,\ell_{m}^{(2)},m)<\frac{1}{4}$, ce qui est absurde. Par conséquent, pour $i=1,2,$  $\mathrm{CS}(\F,\ell_{m}^{(i)},m)\not\in\R^{+}$; en particulier, $\mathrm{CS}(\F,\ell_{m}^{(i)},m)\not\in\Q_{>0}$, ce qui montre que $m$ est une singularité simple~de~$\F.$ Le point \textbf{\textit{2.}} est ainsi prouvé.
\smallskip

\noindent Concernant le troisième point, toute singularité non radiale $s$ de $\F$ étant simple, le feuilletage $\F$ ne peut avoir aucune autre courbe invariante passant par $s$ hormis les deux droites $\ell_{s}^{(1)}$ et $\ell_{s}^{(2)}$ (\emph{cf.} \cite[Corollaire~3.8]{CCD13}). La~courbe~$\Ccal$ n'étant pas une droite, il en résulte que toute singularité de $\F$ sur $\Ccal$ est~forcément~radiale.
\end{proof}

\begin{proof}[\sl D\'emonstration du Théorème~\ref{theoreme:C-invariante-convexe-reduit-plat}]
Le cas où la courbe $\Ccal$ est formée uniquement de droites ($\F$-invariantes) fait l'objet de \cite[Théorème~1]{Bed24arxiv}. Supposons donc que $\Ccal$ ait au moins une composante irréductible qui~n'est~pas~une droite. Si $\Xi_{\F}^{\mathcal{C}}$ est vide alors $\Leg\pref$ est plat par application du Corollaire~\ref{cor:platitude-Leg-pref-convexe-reduit}. Notons que la condition $\Xi_{\F}^{\mathcal{C}}=\vide$ est toujours vérifiée si $\Ccal$ contient au plus une droite. En effet, désignons par $\Ccal_0$ l'union des composantes irréductibles de $\Ccal$ qui ne sont pas des droites. D'après le point~\textbf{\textit{3.}} de la Proposition~\ref{pro:F-convexe-reduit-singularite-simple}, nous avons $\Sigma_{\F}^{\mathcal{C}_0}\subset\Sigma_{\F}^{\mathrm{rad}}$; si $\Ccal=\Ccal_0$ ou $\Ccal=\Ccal_0\cup\ell,$ avec $\ell$ une droite, alors $\Sing\mathcal{C}\subset\Sigma_{\F}^{\mathcal{C}_0}\subset\Sigma_{\F}^{\mathrm{rad}},$~d'où~$\Xi_{\F}^{\mathcal{C}}=\vide.$
\smallskip

\noindent Considérons donc le cas où $\Ccal=(\cup_{i=1}^{n}\ell_i)\cup\Ccal_0$, où les $\ell_i$ sont des droites $\F$-invariantes, avec $n\geq2.$ Supposons dans un premier temps que $n=2$ et $\Xi_{\F}^{\mathcal{C}}\neq\vide.$ Posons $s:=\ell_1\cap\ell_2$; alors $\Sing\mathcal{C}\subset\{s\}\cup\Sigma_{\F}^{\mathcal{C}_0}\subset\{s\}\cup\Sigma_{\F}^{\mathrm{rad}}.$ Il vient que $\Xi_{\F}^{\mathcal{C}}=\{s\}$; de plus $s\not\in\Ccal_0$ (car sinon $s\in\Sigma_{\F}^{\mathrm{rad}}$), de sorte que $\nu(\Ccal,s)=2.$ En vertu du Lemme~\ref{lem:Leg-pref-s-check}, le $d$-tissu $\Leg\pref$ peut se décomposer près de $\check{s}$ sous la forme $\Leg\pref=\W_2\boxtimes\W_1\boxtimes\W_{d-3}$, où $\W_2=\Leg\ell_1\boxtimes\Leg\ell_2$, $\W_1$ est un feuilletage admettant $\check{s}$ comme courbe invariante et $\W_{d-3}$ est un $(d-3)$-tissu régulier et transverse à $\check{s}.$ En reprenant un argument de la démonstration de~\cite[Proposition~2.2]{Bed24arxiv}, nous obtenons que $K(\Leg\pref)$ a au plus des pôles simples le long de $\check{s}.$ Comme la $2$-forme $K(\Leg\pref)$ est holomorphe sur $\pd\setminus\{\check{s}\}$ (Corollaire~\ref{cor:platitude-Leg-pref-convexe-reduit}) et~que~le diviseur canonique de $\pp$ est de degré~$-3$, nous en déduisons que $K(\Leg\pref)$ est identiquement nulle.

\noindent Supposons maintenant que $n\geq3.$ En posant $\Ccal_i=\ell_i\cup\Ccal_0$ et $\Ccal_{i,j}=\ell_i\cup\ell_j\cup\Ccal_0$, \cite[Lemme~2.1]{Bed24arxiv} assure que
\begin{small}
\begin{align*}
K(\Leg\pref)=K(\boxtimes_{i=1}^{n}\Leg\ell_i)-(n-2)\sum_{i=1}^{n}K(\Leg(\Ccal_i\boxtimes\F))+
\sum_{1\leq i<j\leq n}K(\Leg(\Ccal_{i,j}\boxtimes\F))+\binom{n-1}{2}K(\Leg(\Ccal_0\boxtimes\F)).
\end{align*}
\end{small}
\hspace{-1mm}Or, d'une part, $K(\boxtimes_{i=1}^{n}\Leg\ell_i)\equiv0$, car $\boxtimes_{i=1}^{n}\Leg\ell_i$ est formé de pinceaux de droites; d'autre part, d'après ce qui précède, $K(\Leg(\Ccal_0\boxtimes\F))=K(\Leg(\Ccal_i\boxtimes\F))=K(\Leg(\Ccal_{i,j}\boxtimes\F))\equiv0.$ Il en résulte que $K(\Leg\pref)\equiv0.$
\end{proof}

\section{Exemples de pré-feuilletages convexes réduits}
\bigskip

\noindent Dans ce paragraphe, nous allons donner des exemples de pré-feuilletages convexes réduits dont le feuilletage associé est donné par la Table~1 de \cite{MP13}, à savoir: le feuilletage de \textsc{Fermat} $\F_{0}^{d}$ de degré $d$, le pinceau de \textsc{Hesse} $\Hesse$ de degré~$4$, le feuilletage modulaire de \textsc{Hilbert} $\Hilbertcinq$ de degré~$5$ et le feuilletage de Hesse $\Hessesept$ de degré~$7.$

\subsection{Pré-feuilletages convexes réduits dont le feuilletage associé est $\F_{0}^{d}$}\label{subsec:Fermat}

Le feuilletage de \textsc{Fermat} $\F_{0}^{d}$ est défini en coordonnées homogènes par la $1$-forme
\[
\Omegaoverline_{0}^{d}=x^{d}\alpha+y^{d}\beta+z^{d}\gamma,
\]
où $\alpha=y\mathrm{d}z-z\mathrm{d}y$,\, $\beta=z\mathrm{d}x-x\mathrm{d}z$\, et\, $\gamma=x\mathrm{d}y-y\mathrm{d}x.$ Les $3d$ droites invariantes de $\F_{0}^{d}$ sont:
\begin{small}
\begin{align*}
&
\ell_1=\{x=0\},\hspace{2.5mm}
\ell_2=\{y=0\},\hspace{2.5mm}
\ell_3=\{z=0\},\hspace{2.5mm}
\ell_{k+4}=\{y=\zeta^k x\},\hspace{2.5mm}
\ell_{k+d+3}=\{y=\zeta^k z\},\hspace{2.5mm}
\ell_{k+2d+2}=\{x=\zeta^k z\},
\end{align*}
\end{small}
\hspace{-0.75mm}où $k\in\{0,\ldots,d-2\}$ et $\zeta=\exp(\tfrac{2\mathrm{i}\pi}{d-1}).$ De plus, $\F_{0}^{d}$ admet l'intégrale première rationnelle $\dfrac{z^{d-1}\left(y^{d-1}-x^{d-1}\right)}{y^{d-1}\left(x^{d-1}-z^{d-1}\right)}$; ses feuilles distinctes des droites $\ell_i$ sont les courbes algébriques $(\mathcal{C}_{\lambda})_{\lambda\in\C\setminus\{0,1\}}$ d'équation $P_{\lambda}(x,y,z)=0$, où
\[
\hspace{-0.45cm}
P_{\lambda}(x,y,z)=(xy)^{d-1}-\lambda(xz)^{d-1}+(\lambda-1)(yz)^{d-1}.
\]
Ainsi, toute courbe algébrique $\mathcal{C}$ invariante par $\F_{0}^{d}$ est de la forme $\mathcal{C}=\left(\bigcup\limits_{i\in \Lambda}\ell_i\right)\bigcup\left(\bigcup\limits_{j=1}^{n}\mathcal{C}_{\lambda_j}\right)$, avec $\Lambda\subset\{1,2,\ldots,3d\}$, $n\in\N$ et $\lambda_{j}\neq\lambda_{j'}$ si $j\neq j'$; le tissu $\Leg(\mathcal{C}\boxtimes\F_{0}^{d})$ est plat par le Théorème~\ref{theoreme:C-invariante-convexe-reduit-plat}.
\medskip

\noindent En vertu de~\cite[Exemple~6.5]{BM18Bull} et de \cite[Proposition~4.4]{BM24Pisa}, l'adhérence de \textsc{Zariski} de la $\mathrm{Aut}(\pp)$-orbite de $\F_{0}^{d}$ contient les feuilletages $\mathcal{H}_{0}^{d}$, resp. $\mathcal{H}_{1}^{d}$, resp. $\F_{1}^{d}$ (nécessairement convexes) définis en carte affine $z=1$ par les $1$-formes
\begin{align*}
&\omega_{\hspace{0.2mm}0}^{\hspace{0.2mm}d}=(d-1)y^{d}\mathrm{d}x+x(x^{d-1}-dy^{d-1})\mathrm{d}y,&&
\text{resp.}\hspace{1.5mm}\omega_{\hspace{0.2mm}1}^{\hspace{0.2mm}d}=y^d\mathrm{d}x-x^d\mathrm{d}y,&&
\text{resp.}\hspace{1.5mm}\omegaoverline_{1}^{d}=y^{d}\mathrm{d}x+x^{d}(x\mathrm{d}y-y\mathrm{d}x).
\end{align*}
Autrement dit, nous avons l'inclusion suivante
\begin{align*}
&
\mathcal{O}(\mathcal{H}_{0}^{d})\cup\mathcal{O}(\mathcal{H}_{1}^{d})\cup\mathcal{O}(\F_{0}^{d})\cup\mathcal{O}(\F_{1}^{d})\subset\overline{\mathcal{O}(\F_{0}^{d})}.
\end{align*}
Les feuilletages homogènes $\mathcal{H}_{0}^{d}$ et $\mathcal{H}_{1}^{d}$ sont linéairement conjugués pour $d=2$, mais ce n'est plus le cas pour $d\geq3$, \emph{voir} \cite{BM18Bull}. En outre, en vertu de~\cite[Remarque~4.3~et~Proposition~4.4]{BM24Pisa}, nous avons les égalités suivantes
\begin{align*}
\overline{\mathcal{O}(\mathcal{H}_{0}^{d})}=\mathcal{O}(\mathcal{H}_{0}^{d})\cup\mathcal{O}(\F_{1}^{d})
&&\text{et}&&
\overline{\mathcal{O}(\mathcal{H}_{1}^{d})}=\mathcal{O}(\mathcal{H}_{1}^{d})\cup\mathcal{O}(\F_{1}^{d}).
\end{align*}

\begin{pro}\label{pro:H0-H1-F1}
{\sl Soit $\F\in\{\mathcal{H}_{0}^{d},\mathcal{H}_{1}^{d},\F_{1}^{d}\}$. Pour toute courbe algébrique $\mathcal{C}$ invariante par $\F,$ le tissu $\Leg(\mathcal{C}\boxtimes\F)$ est plat.}
\end{pro}

\begin{proof}[\sl D\'emonstration]
Avec les notations ci-dessus, les feuilletages $\mathcal{H}_{0}^{d}$, $\mathcal{H}_{1}^{d}$ et $\F_{1}^{d}$ sont donnés respectivement en coordonnées homogènes par les $1$-formes
\begin{align*}
\Omega_{0}^{d}=x^d\alpha+y^d\beta+dz\,y^{d-1}\gamma,&&
\Omega_{1}^{d}=x^d\alpha+y^d\beta,&&
\Omegaoverline_{1}^{d}=y^d\beta+x^d\gamma.
\end{align*}
Ils possèdent respectivement les intégrales premières rationnelles suivantes
\begin{align*}
\frac{y}{z}\left(\left(\frac{y}{x}\right)^{d-1}-1\right),&&
z^{d-1}\left(\frac{1}{y^{d-1}}-\frac{1}{x^{d-1}}\right),&&
\left(\frac{x}{y}\right)^{d-1}+\frac{(d-1)z}{x}.
\end{align*}
De plus, en notant $\mathrm{Inv}(\F)$ l'ensemble des droites $\F$-invariantes, nous avons
\begin{align*}
\mathrm{Inv}(\mathcal{H}_{0}^{d})=\mathrm{Inv}(\mathcal{H}_{1}^{d})=\{\ell_1,\ell_2,\ldots,\ell_{d+2}\}
&&\text{et}&&
\mathrm{Inv}(\mathcal{F}_{1}^{d})=\{\ell_1,\ell_2\}.
\end{align*}

\noindent Les feuilles de $\mathcal{H}_{0}^{d}$, resp. $\mathcal{H}_{1}^{d}$, resp. $\mathcal{F}_{1}^{d}$, autres que les droites $\ell_i$, sont les courbes algébriques d'équation
\begin{align*}
&0=Q_{\lambda}(x,y,z):=y^d-x^{d-1}(y+\lambda\,z),\hspace{-2.7cm}&&\lambda\in\C^*,\\
\text{resp.}\hspace{1.5mm}
&0=R_{\lambda}(x,y,z):=x^{d-1}y^{d-1}+\lambda(y^{d-1}-x^{d-1})z^{d-1},\hspace{-2.7cm}&&\lambda\in\C^*,\\
\text{resp.}\hspace{1.5mm}
&0=S_{\lambda}(x,y,z):=x^d+(d-1)y^{d-1}(z+\lambda\,x),\hspace{-2.7cm}&&\lambda\in\C.
\end{align*}

\noindent Ainsi, toute courbe algébrique $\mathcal{C}$ invariante par $\mathcal{H}_{0}^{d}$, resp. $\mathcal{H}_{1}^{d}$, resp. $\mathcal{F}_{1}^{d}$, a une équation de la forme
\begin{align}
&0=f_0(x,y,z):=x^{\delta_1}y^{\delta_2}z^{\delta_3}\prod_{k=0}^{d-2}(y-\zeta^{k}\,x)^{\delta_{k+4}}\prod_{j=1}^{n}Q_{\lambda_j}(x,y,z),
\label{equa:C-invariante-H0}\\
\text{resp.}\hspace{1.5mm}
&0=f_1(x,y,z):=x^{\delta_1}y^{\delta_2}z^{\delta_3}\prod_{k=0}^{d-2}(y-\zeta^{k}\,x)^{\delta_{k+4}}\prod_{j=1}^{n}R_{\lambda_j}(x,y,z),
\label{equa:C-invariante-H1}\\
\text{resp.}\hspace{1.5mm}
&0=\overline{f}_{1}(x,y,z):=x^{\delta_1}y^{\delta_2}\prod_{j=1}^{n}S_{\mu_j}(x,y,z),
\label{equa:C-invariante-F1}
\end{align}
où $n\in\N$, $\lambda_j\in\C^*$ et $\mu_j\in\C$ avec $\lambda_j\neq\lambda_{j'}$ et $\mu_j\neq\mu_{j'}$ si $j\neq j'$,  et $\delta_i=0$ ou $1$ suivant que le lieu des zéros du facteur correspondant soit contenu ou non dans la courbe $\mathcal{C}.$
\medskip

\noindent Supposons d'abord que $\mathcal{C}$ soit donnée par (\ref{equa:C-invariante-H0}) et posons
\begin{align*}
F_0(x,y,z):=z^{\delta_1}y^{\delta_2}(x-y)^{\delta_3}\prod_{k=0}^{d-2}(y-\zeta^{k}\,z)^{\delta_{k+4}}\prod_{j=1}^{n}P_{\rho_j}(x,y,z),
\quad\text{où}\hspace{1mm} \rho_j=\frac{\lambda_j}{(d-1)\varepsilon}.
\end{align*}
Alors la courbe d'équation $F_0(x,y,z)=0$ est invariante par $\F_{0}^{d}$ et, en considérant la famille d'automorphismes $\varphi_0=[y+\varepsilon\,z:y:x]$, nous avons
\begin{align*}
\lim_{\varepsilon\to 0}
\frac{1}{\varepsilon^{\delta_3+1}y^{n(d-2)}}\varphi_{0}^{*}\left(F_0(x,y,z)\Omegaoverline_{0}^{d}\right)=-f_0(x,y,z)\Omega_{0}^{d}.
\end{align*}
Comme $\Leg\left(\{F_0(x,y,z)=0\}\boxtimes\F_{0}^{d}\right)$ est plat, nous en déduisons que $\Leg(\mathcal{C}\boxtimes\mathcal{H}_{0}^{d})$ l'est aussi.
\smallskip

\noindent Supposons maintenant que $\mathcal{C}$ soit décrite par l'équation (\ref{equa:C-invariante-H1}). Posons
\begin{align*}
F_1(x,y,z):=x^{\delta_1}y^{\delta_2}z^{\delta_3}\prod_{k=0}^{d-2}(y-\zeta^{k}\,x)^{\delta_{k+4}}\prod_{j=1}^{n}P_{\sigma_j}(x,y,z),
\quad\text{où}\hspace{1mm} \sigma_j=\frac{\lambda_j}{\varepsilon^{d-1}}.
\end{align*}
Alors l'équation $F_1(x,y,z)=0$ définit une courbe algébrique invariante par $\F_{0}^{d}$. En outre, nous avons
\begin{align*}
\lim_{\varepsilon\to 0}
\frac{1}{\varepsilon^{\delta_3+1}}\varphi_{1}^{*}\left(F_1(x,y,z)\Omegaoverline_{0}^{d}\right)=f_1(x,y,z)\Omega_{1}^{d},
\quad\text{où}\hspace{1mm}
\varphi_{1}=[x:y:\varepsilon\,z].
\end{align*}
Le tissu $\Leg\left(\{F_1(x,y,z)=0\}\boxtimes\F_{0}^{d}\right)$ étant plat, il en résulte que $\Leg(\mathcal{C}\boxtimes\mathcal{H}_{1}^{d})$ l'est également.
\smallskip

\noindent Supposons enfin que $\mathcal{C}$ soit donnée par (\ref{equa:C-invariante-F1}). Posons
\begin{align*}
\overline{F}_1(x,y,z):=x^{\delta_2}y^{\delta_1}\prod_{j=1}^{n}Q_{\tau_j}(x,y,z),
\quad\text{où}\hspace{1mm} \tau_j=-\left((d-1)\mu_j\varepsilon^{d-1}+1\right)\varepsilon.
\end{align*}
Alors la courbe d'équation $\overline{F}_1(x,y,z)=0$ est invariante par $\mathcal{H}_{0}^{d}$ et il est aisé de vérifier que
\begin{align*}
\lim_{\varepsilon\to 0}
\frac{1}{\varepsilon^{\delta_1+(n+1)d}}\overline{\varphi}_{1}^{*}\left(\overline{F}_1(x,y,z)\Omega_{0}^{d}\right)=(1-d)\overline{f}_1(x,y,z)\Omegaoverline_{1}^{d},
\quad\text{où}\hspace{1mm}
\overline{\varphi}_{1}=\big[y:\varepsilon\,x:x+(d-1)\varepsilon^{d-1}z\big]\hspace{0.5mm};
\end{align*}
la platitude de $\Leg\left(\{\overline{F}_1(x,y,z)=0\}\boxtimes\mathcal{H}_{0}^{d}\right)$, déjà établie, entraîne celle de $\Leg(\mathcal{C}\boxtimes\mathcal{F}_{1}^{d})$.
\end{proof}

\subsection{Pré-feuilletages convexes réduits dont le feuilletage associé est $\F_{H}^{4}$ ou $\F_{H}^{7}$}\label{subsec:Hesse-4-7}

Le feuilletage de \textsc{Hesse} $\Hesse$ de degré $4$ est défini en coordonnées homogènes par la $1$-forme
\[
\Omegahesse=yz(2\,x^3-y^3-z^3)\mathrm{d}x+xz(2y^3-x^3-z^3)\mathrm{d}y+xy(2z^3-x^3-y^3)\mathrm{d}z.
\]
\noindent Il possède l'intégrale première rationnelle $\frac{x^3+y^3+z^3}{3xyz}$; autrement dit, il s'agit du pinceau de cubiques de \textsc{Hesse}
\begin{align*}
\mathcal{C}_{\lambda,\mu}\hspace{1mm}\colon\mu(x^3+y^3+z^3)-3\lambda\,xyz=0,\qquad[\lambda:\mu]\in\mathbb{P}^{1}_{\C}.
\end{align*}
Les~$12=3\cdot4$ droites invariantes de $\Hesse$ s'obtiennent en prenant $[\lambda:\mu]\in E:=\{[1:0],[1:1],[\mathrm{j}:1],[\mathrm{j}^2:1]\},$ où $\mathrm{j}=\mathrm{e}^{2\mathrm{i}\pi/3}$:
\begin{small}
\begin{align*}
&\ell_1=\{x=0\},&&\ell_4=\{x+y+z=0\},                        &&\ell_7=\{\mathrm{j}\,x+y+z=0\},&&\ell_{10}=\{\mathrm{j}^2x+y+z=0\},\\
&\ell_2=\{y=0\},&&\ell_5=\{\mathrm{j}^2x+\mathrm{j}\,y+z=0\},&&\ell_8=\{x+\mathrm{j}\,y+z=0\},&&\ell_{11}=\{x+\mathrm{j}^2y+z=0\},\\
&\ell_3=\{z=0\},&&\ell_6=\{\mathrm{j}\,x+\mathrm{j}^2y+z=0\},&&\ell_9=\{x+y+\mathrm{j}\,z=0\},&&\ell_{12}=\{x+y+\mathrm{j}^2z=0\}.
\end{align*}
\end{small}
\hspace{-1.5mm}Si $[\lambda:\mu]\not\in E$, alors $\alpha:=\frac{\lambda}{\mu}\in\C\setminus\{1,\mathrm{j},\mathrm{j}^2\}$ et la cubique $\mathcal{C}_{\lambda,\mu}=\mathcal{C}_{\alpha,1}$ est irréductible. Ainsi, toute courbe algébrique $\mathcal{C}$ invariante par $\Hesse$ est de la forme $\mathcal{C}=\left(\bigcup\limits_{i\in \Lambda}\ell_i\right)\bigcup\left(\bigcup\limits_{k=1}^{n}\mathcal{C}_{\alpha_k,1}\right)$, où $\Lambda\subset\{1,2,\ldots,12\}$, $n\in\N$, $\alpha_k\in\C\setminus\{1,\mathrm{j},\mathrm{j}^2\}$ avec $\alpha_{k}\neq\alpha_{k'}$ si $k\neq k'$; le tissu $\Leg(\mathcal{C}\boxtimes\Hesse)$ est plat par le Théorème~\ref{theoreme:C-invariante-convexe-reduit-plat}.
\smallskip

\noindent Quant au feuilletage de \textsc{Hesse} $\Hessesept$ de degré $7$, il est défini par la $1$-forme
\[
\Omegahessesept=yz(y^3-z^3)(7x^3+y^3+z^3)\mathrm{d}x+xz(z^3-x^3)(x^3+7y^3+z^3)\mathrm{d}y+xy(x^3-y^3)(x^3+y^3+7z^3)\mathrm{d}z.
\]
\noindent Ses $21=3\cdot7$ droites invariantes comprennent les $12$ droites $\Hesse$-invariantes $\ell_1,\ldots,\ell_{12}$ et les $9$ droites suivantes
\begin{small}
\begin{align*}
&\hspace{-1.96cm}\ell_{13}=\{y=x\},            &&\ell_{16}=\{y=z\},            &&\ell_{19}=\{x=z\},\\
&\hspace{-1.96cm}\ell_{14}=\{y=\mathrm{j}\,x\},&&\ell_{17}=\{y=\mathrm{j}\,z\},&&\ell_{20}=\{x=\mathrm{j}\,z\},\\
&\hspace{-1.96cm}\ell_{15}=\{y=\mathrm{j}^2x\},&&\ell_{18}=\{y=\mathrm{j}^2z\},&&\ell_{21}=\{x=\mathrm{j}^2z\}.
\end{align*}
\end{small}
\hspace{-1mm}Notons que~$\Hessesept$ admet l'intégrale première rationnelle $\dfrac{P^3}{(x^3-y^3)(y^3-z^3)(z^3-x^3)Q^3},$ où
\begin{Small}
\begin{align*}
P&=248\,x^6y^6z^6(x^6+y^6+z^6)-154\,x^6y^6z^6(x^3y^3+x^3z^3+y^3z^3)-52\,x^3y^3z^3(x^{12}y^3+x^{12}z^3+x^3y^{12}+x^3z^{12}+y^{12}z^3+y^3z^{12})\\
&\hspace{3.43mm}+14\,x^3y^3z^3(x^{15}+y^{15}+z^{15})-10\,x^3y^3z^3(x^9y^6+x^9z^6+x^6y^9+x^6z^9+y^9z^6+y^6z^9)+6(x^{12}y^{12}+x^{12}z^{12}+y^{12}z^{12})\\
&\hspace{3.43mm}+4(x^{15}y^9+x^{15}z^9+x^9y^{15}+x^9z^{15}+y^{15}z^9+y^9z^{15})+x^{18}y^6+x^{18}z^6+x^6y^{18}+x^6z^{18}+y^{18}z^6+y^6z^{18}
\\
&\hspace{-0.88cm}{\fontsize{11}{11pt}\text{et}}
\\
Q&:=\mathrm{I}_{\Hessesept}=xyz(x^3-y^3)(y^3-z^3)(z^3-x^3)(x^3+y^3+z^3-3xyz)(x^3+y^3+z^3-3\mathrm{j}\,xyz)(x^3+y^3+z^3-3\mathrm{j}^2xyz),
\end{align*}
\end{Small}
\hspace{-1mm}{\it i.e.} $\Hessesept$ n'est rien d'autre que le pinceau $\mathcal{P}$ de courbes de degré $72$ défini par
\begin{align*}
\Gamma_{\lambda,\mu}\hspace{1mm}\colon\mu\,P^3-\lambda(x^3-y^3)(y^3-z^3)(z^3-x^3)Q^3=0,\qquad[\lambda:\mu]\in\mathbb{P}^{1}_{\C}.
\end{align*}
Si $\mathcal{C}$ est une courbe algébrique invariante par $\Hessesept$, alors $\mathcal{C}\subset\mathcal{P}$, et le tissu $\Leg(\mathcal{C}\boxtimes\Hessesept)$ est plat (Théorème~\ref{theoreme:C-invariante-convexe-reduit-plat}).

\subsection{Pré-feuilletages convexes réduits dont le feuilletage associé est $\F_{H}^{5}$}\label{subsec:Hilbert5}

Le feuilletage modulaire de \textsc{Hilbert} $\Hilbertcinq$ de degré~$5$ est défini en coordonnées homogènes par la $1$-forme
\begin{small}
\begin{align*}
\Omegahilbertcinq=\big(x^2-z^2\big)\big(x^2-(\sqrt{5}-2)^2z^2\big)\big(x+\sqrt{5}y\big)\big(y\mathrm{d}z-z\mathrm{d}y\big)+\big(y^2-z^2\big)\big(y^2-(\sqrt{5}-2)^2z^2\big)\big(y+\sqrt{5}x\big)\big(z\mathrm{d}x-x\mathrm{d}z\big).
\end{align*}
\end{small}
\hspace{-0.8mm}Notons que les singularités non radiales de $\Hilbertcinq$ sont à indice de \textsc{Camacho}-\textsc{Sad} $-\frac{3}{2}\pm\frac{\sqrt{5}}{2}\not\in\Q$; ceci implique que $\Hilbertcinq$ n'a pas d'intégrale première rationnelle, contrairement aux exemples précédents. En fait, ce feuilletage n'a pas d'intégrale première de type \textsc{Liouville}, car il n'admet pas de structure transversalement affine (\emph{cf.}~\cite[Théorème~1]{MendesP05}). De plus, d'après~\cite[Théorème~2]{MendesP05}, les seules courbes algébriques invariantes par $\Hilbertcinq$ sont ses quinze droites invariantes; elles sont données par
\begin{SMALL}
\begin{align*}
z\Big(x^2-y^2\Big)\Big(x^2-z^2\Big)\Big(y^2-z^2\Big)\Big(x^2-(\sqrt{5}-2)^2z^2\Big)\Big(y^2-(\sqrt{5}-2)^2z^2\Big)\Big((x+\phi\,y)^2-(\phi-1)^2z^2\Big)\Big(\big(x+(\phi-1)y\big)^2-(\phi-2)^2z^2\Big)=0,
\end{align*}
\end{SMALL}
\hspace{-1.15mm}où $\phi=\frac{1+\sqrt{5}}{2}.$ Si $\mathcal{C}$ est une courbe formée d'un certain nombre $n\leq 15$ de ces droites, alors le $(n+5)$-tissu $\Leg(\mathcal{C}\boxtimes\Hilbertcinq)$ est plat (Théorème~\ref{theoreme:C-invariante-convexe-reduit-plat}).

\section{Conjectures}
\bigskip

\noindent Nous concluons cet article en proposant deux conjectures. La première porte sur la classification des feuilletages~convexes~de~$\pp$, et la seconde affirme que le tissu dual de tout pré-feuilletage convexe $\pref$ sur~$\pp$~est~plat,~ce~qui étendrait le Théorème~\ref{theoreme:C-invariante-convexe-reduit-plat} au cas où $\pref$ n'est pas nécessairement convexe réduit.
\begin{conj}\label{conj:1}
{\sl Soit $\F$ un feuilletage convexe de degré $d\geq2$ sur $\pp$. Alors:
\begin{itemize}
\item ou bien $\F$ est convexe réduit;
\item ou bien $\F$ est homogène;
\item ou bien $\F$ est linéairement conjugué au feuilletage $\F_{1}^{d}$ défini par la $1$-forme $\omegaoverline_{1}^{d}=y^{d}\mathrm{d}x+x^{d}(x\mathrm{d}y-y\mathrm{d}x).$
\end{itemize}
}
\end{conj}

\begin{conj}\label{conj:2}
{\sl Soit $\pref$ un pré-feuilletage convexe de degré $d\geq3$ sur $\pp.$ Alors le $d$-tissu $\Leg\pref$ est plat.}
\end{conj}

\noindent La classification des feuilletages convexes de degré $d=2$ (\emph{cf.} \cite[Proposition~7.4]{FP15} ou \cite[Théorème~A]{BM20Bull}) et~$d=3$ (\cite[Corollaire~C]{BM21Four}) montre la validité de la Conjecture~\ref{conj:1} pour $d\in\{2,3\}.$ Comme $\Leg(\mathcal{C}\boxtimes\F_{1}^{d})$ est plat pour toute courbe algébrique $\mathcal{C}$ invariante par $\F_{1}^{d}$ (Proposition~\ref{pro:H0-H1-F1}) et que $\Leg\preh$ l'est également pour tout pré-feuilletage homogène convexe $\preh$ sur $\pp$ (\cite[Théorème~2]{Bed24arxiv}), la Conjecture~\ref{conj:1} implique la Conjecture~\ref{conj:2} dans le cas des pré-feuilletages convexes de co-degré $k\in\{0,1\},$ {\it i.e.} dans le cas des feuilletages convexes et~des~pré-feuilletages convexes de co-degré $1.$

\bigskip
\noindent\textit{\textbf{Remerciements.  --- }}
Je tiens à remercier Jorge Vitório Pereira pour m'avoir indiqué l'intégrale première rationnelle du feuilletage $\Hessesept$ présentée au \S\ref{subsec:Hesse-4-7}, ainsi que certaines propriétés du feuilletage $\Hilbertcinq$ mentionnées au~\S\ref{subsec:Hilbert5}.


\end{document}